\documentclass{amsart}
\newtheorem{theorem}{Theorem}[section]
\numberwithin{equation}{section}
\usepackage{graphicx}

\begin{document}
\title[]{Stability of PID-Controlled Linear Time-Delay Feedback Systems}

\author{Gianpasquale Martelli}
\address{Via Domenico da Vespolate 8, 28079 Vespolate, Italy}
\curraddr{}
\email{gianpasqualemartelli@libero.it}
\thanks{}



\date{February 15, 2008}


\begin{abstract}
The stability of feedback systems consisting of linear time-delay plants and PID controllers has been investigated for many years by means of several methods, of which the Nyquist criterion, a generalization of the Hermite-Biehler Theorem, and the root location method are well known. The main purpose of these researches is to determine the range of controller parameters that allow stability.  Explicit and complete expressions of the boundaries of these regions and computation procedures with a finite number of steps are now available only for first-order plants, provided with one time delay. In this note, the same results, based on Pontryagin's studies, are presented for arbitrary-order plants.
\end{abstract}

\maketitle

\section{Introduction}

The feedback structure considered in this note is depicted in Fig. 1 and the related transfer functions  of the process $P(s)$ and the controller $C(s)$ are given by
\begin{equation}\label{eq:1.1}
P(s)= K \frac{P_{n}(s)}{P_{d}(s)}\ e^{-L \, s}=K  \frac{\prod_ {i=1}^{i=m}(1+Z_{i}s)}{\prod_ {i=1}^{i=n} (1+T_{i}s)}\ e^{-L \, s}
\end{equation}
\begin{equation}\label{eq:1.2}
C(s)=K_{p}+ \frac{K_{i}}{s}\ + K_{d}s ,
\end{equation}
where $K$ is the plant steady-state gain, $T_{i}$ and $Z_{i}$ the plant time constants, $L$ is the positive plant time delay and $K_{p}$, $K_{i}$ and $K_{d}$ are the parameters of the PID controller.
\begin{figure}[htbp]
\centering
\includegraphics{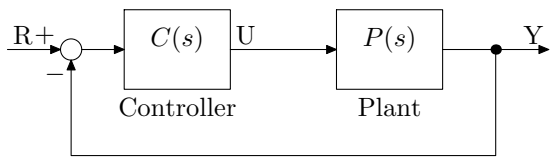}
\caption{Feedback control system}
\end{figure}

Complete explicit expressions of the boundaries of the stability regions in first-order plants have been found in \cite{bib1} with a version of the Hermite-Biehler Theorem derived by Pontryagin, in \cite{bib2} with the Nyquist criterion, and in \cite{bib3} with the root location method.
Moreover the second-order plants have been investigated in \cite{bib4} by means of a graphical approach; the results obtained are correct, but the stability conditions are not all explicit and no finite number of required computation steps is specified.
Finally arbitrary-order plants have been studied with the Nyquist criterion  in \cite{bib5}, but P  and PID controllers with a given $K_{p}$ separately are considered and no information about the set of the process parameters that allow stability is given. 

This note can be considered as an extension of \cite{bib1} to the arbitrary-order plants and is organized as follows. In Section 2 all the analytical expressions that will be used in the next sections are evaluated in detail. In Section 3 a process transfer function without zeros is considered and the stability regions are  explicitly evaluated by means of a version of the Hermite-Biehler Theorem derived by Pontryagin, already used in \cite{bib1}. A second-order plant is studied as example and the related stability regions are determined and plotted in two figures. In Section 4 a process transfer function with zeros is considered and the stability regions are found by means of a new theorem. The two procedures of the Sections 3 and 4 are essentially equal, consist of a finite number of steps and yield the stability regions in both process and controller parameters planes. In Section 5 some conclusive remarks are given.

The importance of explicit expressions of the boundaries of the stability zones has been enhanced by the introduction of the controller tuning charts in \cite{bib6} (used also in \cite{bib7}).

\section{Preliminaries}
The closed-loop transfer function $T(s)$ of the system is given by
\begin{equation}\label{eq:2.0a}
T(s)= \frac{K (K_{i}+K_{p}s+ K_{d}s^{2}) P_{n}(s)}{s P_{d}(s)e^{L \, s}+K (K_{i}+K_{p}s+K_{d}s^{2})P_{n}(s)}\ .
\end{equation}
According to the Pontryagin's studies, presented in \cite{bib8} and summarized in \cite{bib9}, it is necessary that $T(s)$  has a bounded number of poles  with arbitrary large positive real part for stability. This holds if the denominator of $T(s)$ has a principal term $a_{p \, q}s^{p}e^{q \,s}$ (in our case, where $p=n+1$ and $q=1$, it exists if $m<=n-1$) and the function $\chi_{p}(s)$, coefficient of $s^{p}$, (in our case $\chi_{p}(s)= e^{q \,s} \prod_{i=1}^{i=n}T_{i}$ for $m<n-1$ and $\chi_{p}(s)= e^{q \,s} \prod_{i=1}^{i=n}T_{i}+ K \,K_{d} \prod_{i=1}^{i=m}Z_{i}$ for $m=n-1$) has all the zeros in the open left half plane. This happens if one of the following conditions is satisfied:
\begin{enumerate}
\item[(a)]  $m<n-1$
\item[(b)]  $m=n-1$ and $ \left| K \,K_{d} \prod_{i=1}^{i=m}Z_{i} \right| < \left|  \prod_{i=1}^{i=n}T_{i} \right|$ .
\end{enumerate}
The denominator of $T(s)$, given by (\ref{eq:2.0a}), divided by $P_{n}(s)/L$ and hence named $H(s)$, can be written, according to (\ref{eq:1.1}), as 
\begin{equation}\label{eq:2.0b}
H(s)= L \, s  e^{L \, s} \frac{ \prod_{i=1}^{i=n} (1+T_{i}s) }{\prod_{i=1}^{i=m}(1+Z_{i}s)}\ + L \,K (K_{i}+K_{p}s+K_{d}s^{2}) .
\end{equation} 
Since all the poles of the closed-loop transfer function $T(s)$ are zeros of $H(s)$ and a system is stable if no pole of $T(s)$ lies in the right half-plane, the above system is stable if no zero of $H(s)$ lies in the right half-plane.
For process transfer functions without zeros, examined in Section 3, $H(s)$ is a quasi-polynomial and a version of the Hermite-Biehler Theorem derived by Pontryagin is employed. For process transfer functions with zeros, examined in Section 4, $H(s)$ is a quasi-polynomial divided by a polynomial and a new theorem, proved by use of the  Principle of the Argument, is employed.
 
Now, before the explanation of the proposed procedures, let us evaluate all the expressions that will be used in the next sections. It is convenient to introduce the normalized time referred to the plant time delay $L$  and the dimensionless parameters $\sigma=L \,s$, $t_{i}=T_{i}/L$, $z_{i}=Z_{i}/L$, $h=K \,K_{p}$, $h_{i} = K \,K_{i}L$ and $h_{d} = K \,K_{d}/L$, in order to obtain equations independent of the real values of the parameters.
Applying these simplifications, (\ref{eq:2.0b}) becomes 
\begin{equation}\label{eq:2.1}
H(\sigma)= \sigma  e^{\sigma} \frac{ \prod_{i=1}^{i=n} (1+t_{i} \sigma) }{\prod_{i=1}^{i=m}(1+z_{i} \sigma)}\ + h_{i}+h \, \sigma+h_{d} \sigma^{2} .
\end{equation} 
Moreover, assuming $P_{d}(j \, y/L)=A(y)+jB(y)$ and $P_{n}(j \, y/L)=C(y)+jD(y)$, the real and the imaginary components $F(y)$ and $G(y)$ of  $H(\sigma)$, calculated for $\sigma=j \, y$, are given by
\begin{equation}\label{eq:2.2}
F(y)= h_{e}- F_{1}(y)
\end{equation} 
\begin{equation}\label{eq:2.3}
G(y)= y[h-G_{1}(y)]
\end{equation}
where
\begin{equation}\label{eq:2.4}
F_{1}(y)=y[Q(y)cos(y)+P(y)sin(y)]
\end{equation}
\begin{equation}\label{eq:2.5}
G_{1}(y)=-P(y)cos(y)+Q(y)sin(y)
\end{equation}
\begin{equation}\label{eq:2.6}
h_{e}=h_{i}-h_{d} \,y^{2}
\end{equation}
\begin{displaymath}
P(y)=\frac{A(y)C(y)+B(y)D(y)}{C^{2}(y)+D^{2}(y)}\
\end{displaymath}
\begin{displaymath}
Q(y)=\frac{-A(y)D(y)+B(y)C(y)}{C^{2}(y)+D^{2}(y)}\
\end{displaymath}
\begin{equation}\label{eq:2.9}
A(y)= 1+ \sum_{i=1}^{i=int(n/2)} U(n,2 \, i)(-1)^{i}y^{2 \, i}
\end{equation}
\begin{equation}\label{eq:2.10}
B(y)= \sum_{i=0}^{i=int((n-1)/2)} U(n,2 \, i+1)(-1)^{i}y^{2 \,i+1}
\end{equation}
\begin{equation}\label{eq:2.11}
C(y)= 1+ \sum_{i=1}^{i=int(m/2)} V(m,2 \, i)(-1)^{i}y^{2 \, i}
\end{equation}
\begin{equation}\label{eq:2.12}
D(y)= \sum_{i=0}^{i=int((m-1)/2)} V(m,2 \, i+1)(-1)^{i}y^{2 \, i+1} .
\end{equation}
For sake of clarity, $U$ and $V$ are the symmetric expressions of the time constants $t_{i}$ and $z_{i}$; $U(n,k)$ is the sum of the $\binom{n}{k}$ products of $k$  different $t_{i}$ selected among the total $n$ (for example $U(3,2)=t_{1}t_{2}+t_{2}t_{3}+t_{3}t_{1}$ and $U(3,3)=t_{1}t_{2}t_{3}$).

The derivative of $G(y)$ with respect to $y$ is given by
\begin{equation}\label{eq:2.13}
G'(y)= h-G_{1}(y) - yG'_{1}(y).
\end{equation}

Assuming $G(y)=0$ and $h=-1$ in (\ref{eq:2.3}), one obtains $tan(y/2) = E(y)$ where
\begin{equation}\label{eq:2.14}
E(y)= \frac{-Q(y) \pm \sqrt{P^{2}(y)+Q^{2}(y)-1}}{1+P(y)}\ .
\end{equation}
It is easy to check that $E(y)$ exists for $y=0$ only if $\sum_{i=1}^{i=n} t_{i}^{2} >\sum_{i=1}^{i=m} z_{i}^{2}$.
The derivative of $E(y)$ with respect to $y$, evaluated at $y=0$ and named $E'(0)$, is given by
\begin{equation}\label{eq:2.15}
\begin{split}
&E'(0) = +0.5(-U(n,1)+V(m,1)) \\  &\pm 0.5 \sqrt{U^{2}(n,1)-2 \, U(n,2)-V^{2}(m,1)+2 \, V(m,2)} .
\end{split}
\end{equation}
Denoting by $E_{-}(y)$ and $E_{+}(y)$ the two branches of $E(y)$ related respectively to the minus and plus signs, their derivatives 
$E'_{-}(0)$ and $E'_{+}(0)$ are higher than the derivative of $tan(y/2)$, equal to 0.5, depending on $\Phi_{1}$ and $\Phi_{2}$, given by
\begin{equation}\label{eq:2.16}
\Phi_{1} =1+U(n,1)-V(m,1) 
\end{equation}
\begin{equation}\label{eq:2.17}
\begin{split}
\Phi_{2} = &+1+2 \,U(n,1)+2 \,U(n,2) -2 \,U(n,1)V(m,1) \\ &+2 \,V^{2}(m,1)-2 \,V(m,1)-2 \,V(m,2).
\end{split}
\end{equation}
In detail, the number of the derivatives $E'_{-}(0)$ and $E'_{+}(0)$ higher than 0.5 are the following: zero if $\Phi_{1}>0$ and $\Phi_{2}>0$, one if  $\Phi_{2}<0$, and two if $\Phi_{1}<0$ and $\Phi_{2}>0$.
Let us denote by $E_{d}$ 
\begin{equation}\label{eq:2.19}
E_{d}=tan(y_{t}/2)- E(y_{t})
\end{equation}
where $y_{t}$, corresponding to equal derivatives with respect to $y$ of $tan(y/2)$ and $E(y)$, is a root of $0.5(1+tan^{2}(y_{t}/2))= E'(y_{t})$.

Differentiating (\ref{eq:2.5}) with respect to $y$ once and twice, one obtains
\begin{equation}\label{eq:2.20}
\frac{dG_{1}(y)}{dy}\ =(-P'(y)+Q(y))cos(y)+(P(y)+Q'(y))sin(y)
\end{equation}
\begin{equation}\label{eq:2.21}
\begin{split}
\frac{d^{2}G_{1}(y)}{dy^{2}}\ = &+(+P(y)-P''(y)+2 \,Q'(y))cos(y)\\ &+(-Q(y)+Q''(y)+2 \,P'(y))sin(y) .
\end{split}
\end{equation}
It is worthwhile to note that $G_{1}(0)=-1$, $dG_{1}(0)/dy=0$, and $d^{2}G_{1}(0)/dy^{2}=\Phi_{2}$, where $\Phi_{2}$ is given by (\ref{eq:2.17}).
Evaluating $h_{m}=\left|G_{1}(y_{m})\right|$, where $y_{m}$ is a root of $dG_{1}(y)/dy$ given by (\ref{eq:2.20}), one obtains 
\begin{equation}\label{eq:2.22}
h_{m}= \frac{ \left|P^{2}(y_{m})+Q^{2}(y_{m})+P(y_{m})Q'(y_{m})-P'(y_{m})Q(y_{m}) \right|}{\sqrt{(P'(y_{m})-Q(y_{m}))^{2}+(P(y_{m})+Q'(y_{m}))^{2}}}\ .
\end{equation}

Eliminating $cos(y)$ and $sin(y)$ from $F(y)=0$ and $G(y)=0$, given by  (\ref{eq:2.2}) and (\ref{eq:2.3}), yields
\begin{equation}\label{eq:2.23}
h_{e}= y \sqrt{P^{2}(y)+Q^{2}(y)-h^2} .
\end{equation}
Denote by $\Gamma_{ai}$ and $\Gamma_{bi}$ the two straight lines whose equations in the ($h_{i}, h_{d}$)-plane  are obtained introducing in (\ref{eq:2.6}) respectively $y=y_{ai}$, $h=h_{eai}$ and $y=y_{bi}$, $h=h_{ebi}$ (see Figs. 2 and 5); denote further by $V_{i}$, $U_{i}$ and $W_{i}$ the vertices of a triangle, whose sides are the axis $h_{d}$ and the two lines $\Gamma_{ai}$ and $\Gamma_{bi}$. The coordinates of these vertices are given by
\begin{equation}\label{eq:2.24}
h_{i}(V_{i})= \frac{y_{bi}^{2}h_{eai}-y_{ai}^{2}h_{ebi}}{y^{2}_{bi}-y^{2}_{ai}}\
; \quad h_{d}(V_{i})=  \frac{-h_{ebi}+h_{eai}}{y^{2}_{bi}-y^{2}_{ai}}\
\end{equation}
\begin{equation}\label{eq:2.25}
h_{d}(U_{i})= - h_{ebi}/y_{bi}^{2} ; \quad h_{d}(W_{i})= - h_{eai}/y_{ai}^{2} . 
\end{equation}
Considering (\ref{eq:2.23}), the coordinates $h_{d}(R_{i})$ and $h_{d}(S_{i})$ of the points lying on $\Gamma_{ai}$ and $\Gamma_{bi}$ at $h_{i}=h_{i}(V_{1})$ are given by
\begin{equation}\label{eq:2.26}
h_{d}(R_{i})= \frac{h_{i}(V_{1})}{y^{2}_{bi}}\ - sign[h_{ebi}]   \frac{\sqrt{P^{2}(y_{bi})+Q^{2}(y_{bi})-h^{2}}}{y_{bi}}\
\end{equation}
\begin{equation}\label{eq:2.27}
h_{d}(S_{i})= \frac{h_{i}(V_{1})}{y^{2}_{ai}}\ - sign[h_{eai}]    \frac{\sqrt{P^{2}(y_{ai})+Q^{2}(y_{ai})-h^{2}}}{y_{ai}}\ .
\end{equation}
It is easy to check that, when $i= \infty$, the absolute values of $h_{d}(U_{i})$ and $h_{d}(W_{i})$ , if $m<n-1$, are equal to $\infty$ and, if $m=n-1$, to $H_{d}( \infty)$ given by
\begin{equation}\label{eq:2.28}
H_{d}( \infty)= \left|U(n,n)/V(m,m) \right| .
\end{equation}

\section{Process transfer function without zeros}
When the process transfer function does not have zeros, $m=0$ and thus  $C(y)=1$, $D(y)=0$, $P(y)=A(y)$, $Q(y)=B(y)$ hold; therefore, the function $H(\sigma)$, given by (\ref{eq:2.1}), is a quasi-polynomial and the Pontryagin's results are integrally applicable. The following two conditions derived from Theorem 3.2 of \cite{bib1}\ and from Theorem 13.7 of \cite{bib9}, respectively, must be satisfied in order to have a stable system:
\begin{itemize}
\item Condition no. 1\\
Consider that the principal term of $H(\sigma)$, given by (\ref{eq:2.1}), is $\sigma^{n+1}e^{ \sigma}$, set $H(j \,y)=F(y)+jG(y)$ and let $\epsilon$ be an appropriate constant such that the coefficient of  $y^{n+1}$ in $G(y)$ does not vanish at $y=\epsilon$. The number $N_{r}$ of the real roots of $G(y)$ in the interval $-2 \,r \pi + \epsilon \leq y \leq 2 \,r \pi  + \epsilon$ for sufficiently large $r$ must be
\begin{equation}\label{eq:3.1}
N_{r}=4 \, r +n+1 .
\end{equation}

\item Condition no. 2\\
For all the zeros $y=y_{0}$ of the function $G(y)$ the inequality $G'(y_{0})F(y_{0})-G(y_{0})F'(y_{0})>0$, that is $G'(y_{0})F(y_{0})>0$, must hold.
\end{itemize}

In order to study both stable and unstable free-delay plants, the following two cases, adopted also in \cite{bib1}, are considered:
\begin{enumerate}
\item $U(n,n)>0$ (even number of negative plant time constants $T_{i}$) \\ $h>-1$ and $h_{i}>0$.
\item $U(n,n)<0$ (odd number of negative plant time constants $T_{i}$) \\ $h<-1$ and $h_{i}<0$.
\end{enumerate}

From (\ref{eq:2.3}), (\ref{eq:2.9}) and (\ref{eq:2.10}) it follows that the  coefficient of the highest degree of $y$ in $G(y)$ is $U(n,n)cos(y)$ when $n$ is even and $U(n,n)sin(y)$ when odd; hence we assume $ \epsilon= 0$ if $n$ is even and $ \epsilon= 0.5 \pi $ if  odd.
\begin{figure}[htbp]
\centering
\includegraphics{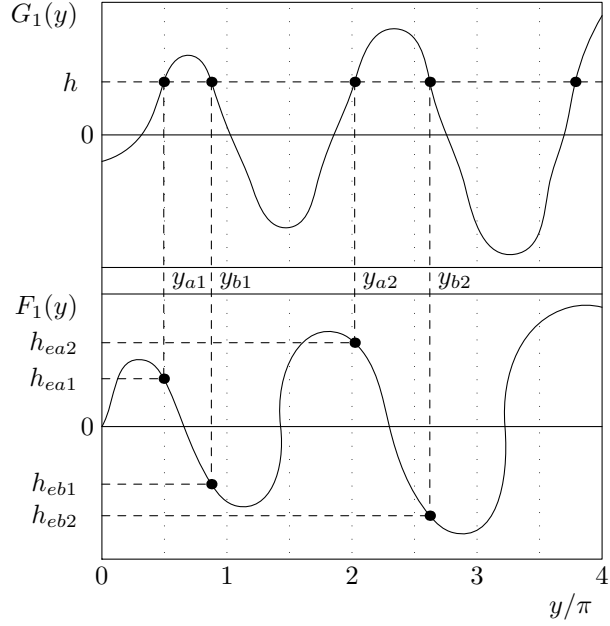}
\caption{Typical plots of $F_{1}(y)$ and $G_{1}(y)$}
\end{figure}

Typical functions $F_{1}(y)$ and $G_{1}(y)$ are plotted in Fig. 2; according to (\ref{eq:2.3}) there is one root of $G(y)$ at $y=0$ and one for each intersection of $G_{1}(y)$ with the horizontal line having the ordinate equal to a given $h$.

Denoting by $N_{e}$ the number of the intersections between $G_{1}(y)$ and $h=-1$ corresponding to $N_{r}$ in Fig. 2 and assuming that no local minimum or maximum of $G_{1}(y)$ is equal to $-1$ for $y \neq 0$, the relationship  between $N_{e}$ and $N_{r}$ is given by
\begin{equation}\label{eq:3.2}
\begin{split}
&N_{e}=N_{r}-1=4 \,r+n \quad for \quad U(n,n) \Phi_{2}<0\\ &N_{e}=N_{r}-3=4 \, r +n-2 \quad for \quad U(n,n) \Phi_{2}>0 ,
\end{split}
\end{equation}
where $\Phi_{2}$ is according to (\ref{eq:2.17}); (\ref{eq:3.2}) can be easily checked considering that from  (\ref{eq:2.5}), (\ref{eq:2.20}) and (\ref{eq:2.21}) it follows $G_{1}(0)=-1$, $G'_{1}(0)=0$ and $G''_{1}(0)=\Phi_{2}$, and also that $h>-1$ must hold if $U(n,n)>0$ and $h<-1$ if $U(n,n)<0$.
Since $h=-1$ is the common limit value of $h$ for the two considered cases ($U(n,n)<0$ and $U(n,n)>0$), the existence of the $N_{e}$ intersections given by (\ref{eq:3.2}) represents a prerequisite of a plant to be made stable.

The number $N_{e}$ can be evaluated by counting the intersections of the plots of $tan(y/2)$ and $E(y)$, given by (\ref{eq:2.14}). From (\ref{eq:2.14}) it follows that $E(y)$ is an odd function of $y$; moreover, assuming $S_{a}=sign[A(+ \infty)]$ and $S_{b}=sign[B(+ \infty)]$, $E_{-}(+ \infty)$ and $E_{+}(+ \infty)$ can be expressed as
\begin{itemize}
\item $n$ even \\ $E_{-}(+ \infty)=- S_{a} \,1$ and $E_{+}(+ \infty)=+ S_{a} \,1$ .
\item $n$ odd \\ $E_{-}(+ \infty)=- S_{a} \, \infty$ for $S_{b}>0$ and $E_{-}(+ \infty)=- S_{a} \, 0$ for $S_{b}<0$, \\ $E_{+}(+ \infty)=+ S_{a} \, 0$ for $S_{b}>0$ and $E_{+}(+ \infty)=+ S_{a} \, \infty$ for $S_{b}<0$.
\end{itemize}
If $E(y)$ has no pole, splitting $N_{e}$ into $N_{e1}$ ($ \left|y \right|< \pi$) and $N_{e2}$ and considering the above described behavior of $E(y)$ at $y= \infty$ and also at $y=0$, one obtains
\begin{itemize}
\item $ \left|y \right|< \pi$ \\
$N_{e1}=0$ if $\Phi_{1}>0$ and $\Phi_{2}>0$, \\ $N_{e1}=2$ if 
 $\Phi_{2}<0$, \\ $N_{e1}=4$ if $\Phi_{1}<0$ and $\Phi_{2}>0$, \\
where $\Phi_{1}$ and $\Phi_{2}$ are given by (\ref{eq:2.16}) and (\ref{eq:2.17}).
\item $-2 \,r \pi + \epsilon \leq y \leq - \pi$; $\pi \leq y \leq 2 \,r \pi  + \epsilon$ (see Fig. 3 (a)) \\
$N_{e2}=4 \,r -2$ if $n$ is even,\\ $N_{e2}=4 \,r -1$ if $n$ is odd and $A(+ \infty)B(+ \infty)>0$,\\ $N_{e2}=4 \,r -3$ if $n$ is odd and $A(+ \infty)B(+ \infty)<0$,\\ where $sign[A(+ \infty)B(+ \infty)]=sign[-(-1)^{n}U(n,n-1)U(n,n)]$ .
\end{itemize}
It is clear that, if $E(y)$ has no pole, $N_{e}$ will be always lower than the value required by (\ref{eq:3.2}) for enough large $n$. A positive solution can be reached only if $E(y)$ is provided with a suitable number of poles, since $N_{e2}$ is increased by one for each added pole (see Fig. 3); this happens if $E_{d}>0$ for case (b1), if $E_{d}<0$ for case (b2) and without further condition for case (b3), where $E_{d}$ is  given by (\ref{eq:2.19}).
Since the denominator of $E(y)$ is a polynomial of $y^{2}$ of degree $d=int(n/2)$, the maximum number of poles of $E(y)$ is equal to $2 \, d$ and the actual number can be determined by means of the Sturm Theorem, as detailed in Appendix A.
 
\begin{figure}[htbp]
\centering
\includegraphics{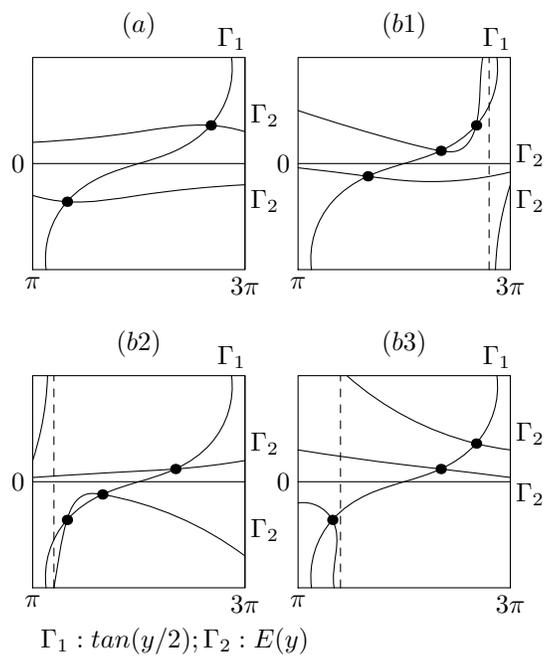}
\caption{Plots of $tan(y/2)$ and $E(y)$ for $y> \pi$}
\end{figure}

The sought-after procedure can be detailed as follows:
\begin{enumerate}
\item Process parameters (see Fig. 4)\\
The stability region in the process parameters plane is that where (\ref{eq:3.2}) holds; its boundary line is a proper set of the boundary lines of the zones with different numbers $N_{e}$, that is,
$\Phi_{1}=0$ and $\Phi_{2}=0$ according to (\ref{eq:2.16}) and (\ref{eq:2.17}) and $\Psi_{i,j}=0$, as explained in Appendix A, and eventually $E_{d}=0$ given by (\ref{eq:2.19}).
Since the expressions of these boundary lines are functions of $t_{i}=T_{i}/L$, it is possible to evaluate the stability range of $L$ when the parameters $T_{i}$ are known and, conversely, of each $T_{i}$ when $L$ and the remaining $T_{i}$ are known.
Moreover, if one needs only to know if a given plant can be made stable, it is not necessary to determine the stability regions, but it is enough to examine the plots of $E(y)$ and $tan(y/2)$ and to check whether the number of the intersections $N_{e}$ satisfies (\ref{eq:3.2}).

\item Controller parameter $h$ (see Fig. 2)\\
The requirement, stated as Condition no. 1, is  fulfilled if the selected value of $h$ is included in the interval from $-1$ to $h_{p}$ for $U(n,n)>0$ or from $h_{n}$ to $-1$ for $U(n,n)<0$, where $h_{p}$ and $h_{n}$  are respectively the local maxima and minima of $G_{1}(y)$ nearest to $-1$; both can be calculated by introducing in (\ref{eq:2.5}) the related root of $dG_{1}(y)/dy$ obtained from (\ref{eq:2.20}). A finite number of these maxima or minima must be examined in order to find the nearest to $-1$, exactly up to the first value of $y_{m}$ higher than $y_{r1}$. This limit $y_{r1}$ is the largest positive root of the derivative with respect to $y_{m}$ of $h_{m}$, since $h_{m}$, given by (\ref{eq:2.22}), monotonically increases for $y_{m}>y_{r1}$. It is obvious that, if such root does not exist, only the first value must be considered.

\item Controller parameters $h_{i}$, $h_{d}$ (see Fig. 5)\\
Considering (\ref{eq:2.2}), (\ref{eq:2.6}) and (\ref{eq:2.13}), the requirement, stated as Condition no. 2, is fulfilled if the following inequalities
\begin{equation}\label{eq:3.3}
\begin{split}
&h_{i}-h_{d}y^{2}_{0}<F_{1}(y_{0}) \quad  if \quad  G'_{1}(y_{0})>0 \\  &h_{i}-h_{d}y^{2}_{0}>F_{1}(y_{0}) \quad  if \quad   G'_{1}(y_{0})<0
\end{split}
\end{equation}
hold for each root $y=y_{0}$ of $G(y)$, given by (\ref{eq:2.3}), evaluated with a value of $h$ included in the above specified interval.

The stability region in the ($h_{i},h_{d}$)-plane consists of the intersection of a finite number of triangles; each of them is related to a couple of roots $y_{ai}$ and $y_{bi}$ of $G(y)$, has the axis $h_{d}$ and the two straight lines given by (\ref{eq:3.3}) as sides and the points $U_{i}$, $V_{i}$ and $W_{i}$ as vertices, whose coordinates are given by (\ref{eq:2.24}) and (\ref{eq:2.25}) (see Figs. 2 and 5). Since, as $i \rightarrow \infty$,  $h_{d}(U_{i})$ and $h_{d}(R_{i})$ approach $ \pm \infty$ and  $h_{d}(W_{i})$ and $h_{d}(S_{i})$ approach $ \mp \infty$, each triangle includes definitely the first one when $y_{bi}>y_{r2}$. This limit $y_{r2}$ is the bigger among $y_{b1}$ and the largest root of the derivatives with respect to $y_{bi}$ of $h_{d}(U_{i})$ and $h_{d}(R_{i})$, given by (\ref{eq:2.25}) and (\ref{eq:2.26}). Therefore, it is not necessary to  examine the triangles for $y_{bi}>y_{r2}$. 
\end{enumerate}

A second-order plant, whose transfer function is without zeros, is considered as example and the stability regions of the plant and the controller parameters are depicted respectively in Figs. 4 and 5. In Fig. 4 this region consists of the following: 
\begin{itemize}
\item $Z_{1}$: $U(n,n)>0$; $ \Phi_{1}>0$ and $ \Phi_{2}>0$
\item $Z_{2}$: $U(n,n)>0$; $ \Phi_{1}<0$ and $ \Phi_{2}>0$
\item $Z_{3}$: $U(n,n)<0$; $ \Phi_{2}<0$
\end{itemize}
where $\Phi_{1}=1+t_{1}+t_{2}$ and $\Phi_{2}=1+2 \,t_{1}+2 \, t_{2}+2 \,t_{1}t_{2}$ according to (\ref{eq:2.16}) and (\ref{eq:2.17}).
The required number $N_{e}$ of the intersections between $E(y)$ and $tan(y/2)$, equal to $4 \, r$ as per (\ref{eq:3.2}), coincide with the actual number only in this zone.

Since $P(y)=1-t_{1}t_{2}y^{2}$ $Q(y)=(t_{1}+t_{2})y$, considering the point ($t_{1}=0.6$; $t_{2}=0.8$) lying in $Z_{1}$ of Fig. 4, one obtains $y_{p}=1.778$ from  (\ref{eq:2.20}) for the first root $y=y_{p}$ of $dG_{1}(0)/dy$  and $h_{p}=G_{1}(y_{p})=2.330$ from (\ref{eq:2.5}). Since $-1<h<h_{p}$ must hold, let us assume $h$ equal to $0.5$;  for the first two roots higher than zero of $G(y)$ one obtains $y_{a1}=0.863$ and $y_{b1}=2.498$ from  (\ref{eq:2.3}) and hence $h_{ea1}=F_{1}(y_{a1})=1.099$ and $h_{eb1}=F_{1}(y_{b1})=-9.985$ from (\ref{eq:2.4}).
From (\ref{eq:2.24}) and (\ref{eq:2.25}) it follows $h_{i}(V_{1})=2.600$, $h_{d}(V_{1})=2.016$, $h_{d}(U_{1})=1.600$ and $h_{d}(W_{1})=-1.476$. Similarly, for the third and fourth roots of $G(y)$ one obtains $y_{a2}=5.285$ and $y_{b2}=8.191$, $h_{ea2}=76.290$, $h_{eb2}=-272.288$, $h_{d}(U_{2})=4.058$ and $h_{d}(W_{2})=-2.732$; moreover $h_{d}(R_{2})=4.097$ and $h_{d}(S_{2})=-2.638$ from  (\ref{eq:2.26}) and (\ref{eq:2.27}).

\begin{figure}[htbp]
\centering
\includegraphics{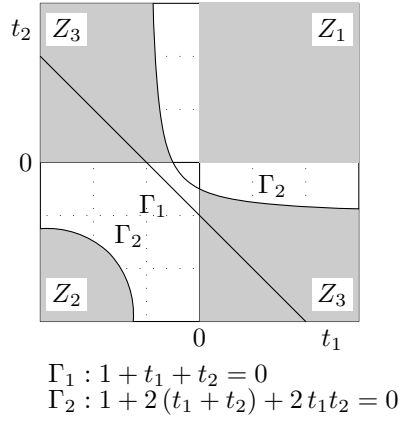}
\caption{Stability zone of the process parameters for $n=2$ and $P_{n}(s)=1$}
\end{figure}

\begin{figure}[htbp]
\centering
\includegraphics{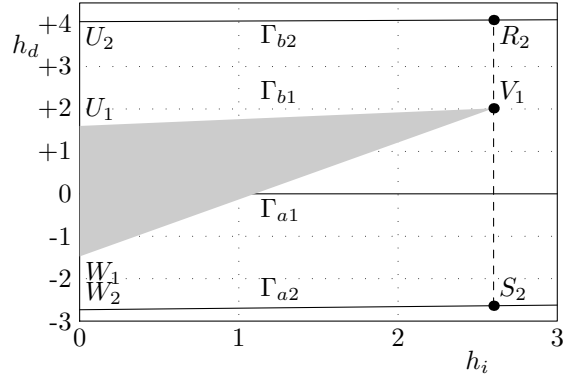}
\caption{Stability zone of the controller parameters $h_{i}$ and $h_{d}$ for $n=2$, $P_{n}(s)=1$, $t_{1}=0.6$, $t_{2}=0.8$, $h=0.5$}
\end{figure}

\section{Process transfer function with zeros}
The function $H(\sigma)$, given by (\ref{eq:2.1}), can be rewritten as
\begin{equation}\label{eq:4.1}
H(\sigma)=  \sigma^{n+1-m}  e^{\sigma} \frac{U(n,n)}{V(m,m)}\ \frac{1+ H_{n}}{ 1+H_{d}}\ ,
\end{equation}
where 
\begin{displaymath}
H_{n}=  \sum_{i=1}^{i=n} \frac{U(n,n-i)}{U(n,n)\sigma^{i}}\ + e^{- \sigma} (h_{i} \sigma^{-2}+ h \sigma^{-1}+h_{d} ) \sum_{i=n-1-m}^{i=n-1} \frac{V(m,n-1-i)}{U(n,n)\sigma^{i}}\ 
\end{displaymath}
\begin{displaymath}
H_{d}= \sum_{i=1}^{i=m} \frac{V(m,m-i)}{V(m,m)\sigma^{i}}\  
\end{displaymath}
$m<n-1$ or $m=n-1$ and $ \left |h_{d}V(m,m) \right |<\left |U(n,n) \right |$, as explained in Section 2.
Since the denominator of $H(\sigma)$ is the numerator of the plant transfer function, the poles of $H(\sigma)$ are the zeros of the plant transfer function ($\sigma_{i}=-1/z_{i}$).
In this case the proposed procedure will be in accordance with Theorem 4.1, a generalization of the theorems applied in Section 3 and of Theorem 13.5 of \cite{bib9}; it will be here enunciated and proved by use of the  Principle of the Argument.
\begin{theorem}
Consider a function $H(\sigma)$ of the form in (\ref{eq:4.1}), set $H(j \,y)=F(y)+jG(y)$ and let $\epsilon$ be an appropriate constant such that the coefficient of $y^{n+1-m}$ in the numerator of $G(y)$ does not vanish at $y=\epsilon$. Assume further that the $m_{p}$ poles with positive real part of $H(\sigma)$ lie all in the rectangle $R$, described by the inequalities $-2r \pi + \epsilon<=y<=2r \pi + \epsilon$, $x>0$ ($\sigma=x+j \, y$). Suppose finally that the function $H(\sigma)$ does not assume the value of zero on the imaginary axis, that is, $H(jy)\neq 0$.

All the zeros  of $H(\sigma)$ lie to the left of the imaginary axis if and only if:
\begin{enumerate}
\item[(a)] The vector $w=H(j \, y)$ for real $y$ ranging from $- \infty$ to $+ \infty$  continually revolves in the positive direction at a positive velocity, that is, the inequality $G'(y_{0})F(y_{0})<0$ is satisfied for each root $y=y_{0}$ of $G(y)$.
\item[(b)] The number $N_{r}$ of the roots of $G(y)$ in the interval $-2r \pi + \epsilon <=y<=2r \pi + \epsilon$ for sufficiently large $r$ is
\begin{equation}\label{eq:4.2}
N_{r}=4 \, r +n+1-m+2 \,m_{p} .
\end{equation}
\end{enumerate}
\end{theorem}

\begin{proof}
Denote by $C_{1}$, $C_{2}$, $C_{3}$ and $C_{4}$ the corners of the rectangle $R$ whose coordinates in the $\sigma$-plane are respectively $\sigma_{1}=0+j(-2r \pi + \epsilon)$, $\sigma_{2}=+ \infty +j(-2r \pi + \epsilon)$, $\sigma_{3}=+ \infty +j(+2r \pi + \epsilon)$ and $\sigma_{4}=0+j(+2r \pi + \epsilon)$; the arguments $\theta(\sigma_{1})$ and $\theta(\sigma_{4})$ of $H(\sigma)$ are given by $\theta(\sigma_{1})=-2r \pi+ \epsilon-0.5(n+1-m) \pi + \eta + \delta_{1}$ and $\theta(\sigma_{4})=+2r \pi + \epsilon+0.5(n+1-m) \pi + \eta+\delta_{4}$, where $\eta= 0$ if $V(m,m)/U(n,n)>0$ or $\eta= \pi$ if $V(m,m)/U(n,n)<0$, and $\delta_{1} \rightarrow 0$ and $\delta_{4} \rightarrow 0$ simultaneously with $1/r$. Denote by $N_{z}$ the number of the zeros of $H(\sigma)$ lying in the rectangle $R$, by $V_{a}$ the variation of the argument of $H(\sigma)$ in the counterclockwise direction as $\sigma$ moves around the contour of $R$ from $C_{1}$ to $C_{4}$ through $C_{2}$ and $C_{3}$, and by $V_{b}$ as $\sigma$ moves directly from $C_{4}$ to $C_{1}$.  Using the Principle of the Argument yields
\begin{equation}\label{eq:4.3}
V_{a}+V_{b}=(N_{z}-m_{p}) 2 \, \pi .
\end{equation}
Since  $N_{z}=0$ for stability and , for $r \rightarrow \infty$,  $V_{b}=- N_{r} \pi$ as per condition (a) and $V_{a}=\theta_{4}- \theta_{1}=+4 \, r \pi+(n+1-m) \pi $, (\ref{eq:4.2}) follows from (\ref{eq:4.3}) and the condition (b) is satisfied. 
\end{proof}

The procedure detailed in Section 3 for process transfer functions without  zeros is fully applicable to process transfer functions with zeros; it is only necessary to replace (\ref{eq:3.1}) with  (\ref{eq:4.2}) and to consider $C(y)$ and $D(y)$ as functions given by (\ref{eq:2.11}) and (\ref{eq:2.12}) instead of $C(y)=1$ and $D(y)=0$. Moreover, for $m=n-1$, the rectangle, according to (\ref{eq:3.3}) for $y_{0}=+ \infty$ and provided with horizontal sides symmetric with respect to  the axis $h_{i}$ at a distance given by (\ref{eq:2.28}), must be considered in the $(h_{i},h_{d})$-plane.

\section{Conclusions}
In this note, both stable and unstable delay-free arbitrary-order plants, provided with one time delay and PID controller, have been examined and the related stability regions in process and controller parameter spaces have been determined by use of the Pontryagin's studies. The proposed procedure, consisting of a finite number of steps, yields explicit expressions of the boundaries of the stability zone for the controller parameters. These results can be implemented in tuning charts, which become a complete tool for the design and the maintenance of control systems.

\appendix{}
\section{Sturm Theorem}
The Sturm Theorem states that the number of real roots of an algebraic equation with real coefficients whose real roots are simple over an interval, the endpoints of which are not roots, is equal to the difference between the numbers of sign changes of the Sturm chains formed for the interval ends.

Given a function $f(x)=f_{0}(x)$ of degree $d$, assume $f_{1}(x)=df_{0}(x)/dx$ and define the Sturm functions by
\begin{displaymath}
f_{i}(x)= -f_{i-2}(x)+f_{i-1}(x) \left [ \frac{f_{i-2}(x)}{f_{i-1}(x)}\ \right ],
\end{displaymath}
where $\left [f_{i-2}(x)/{f_{i-1}(x} \right ]$ is a polynomial quotient. These functions can be written as
\begin{displaymath}
f_{i}(x)= \sum_{j=0}^{j=d-int((i+1)/2)} \psi_{i,j} x^{j} \quad 0 \leq i \leq 2 \,d-1 ,
\end{displaymath}
where $\psi_{i,j}$ depends on the coefficients of $x$ in $f(x)$.

In our case $x=y^{2}$ and $f_{0}(x)=C^{2}( \sqrt{x})+D^{2}( \sqrt{x})+A( \sqrt{x})C( \sqrt{x})+B( \sqrt{x})D( \sqrt{x})$ hold; since the roots must be positive, the required interval is from $x=0$ to $x=+ \infty$ and, therefore, the signs of each Sturm function at these ends are the signs of $\psi_{i,j}$ evaluated respectively for $j=0$ and $j=d-int((i+1)/2)$.


\begin{thebibliography}{9}

\bibitem{bib1}
G.~J. Silva, A.~Datta, and S.~P. Bhattacharyya, ``New Results on the Synthesis of PID Controllers,'' \emph{{IEEE} Trans. Automat. Contr.}, vol.~47, pp. 241--252, Feb. 2002.

\bibitem{bib2}
G.~Martelli, `` Comments on 'New results on the synthesis of PID controllers,''' \emph{{IEEE} Trans. Automat. Contr.},   vol.~50, pp. 1468--1469, Sep. 2005.

\bibitem{bib3}
E.~Malakhovski, and L.~Mirkin, ``On Stability of second-order quasi-polynomials with a single delay,'' \emph{Automatica}, vol.~42, no.~6, pp. 1041--1047, Jun. 2006.

\bibitem{bib4}
D.~Wang, ``Further Results on the Synthesis of PID Controllers,'' \emph{{IEEE} Trans. Automat. Contr.}, vol.~52, pp. 1127--1132, Jun. 2007.

\bibitem{bib5}
H.~Xu, A.~Datta, and S.~P. Bhattacharyya, ``PID Stabilization of LTI Plants with Time-Delay,'' in \emph{Proc. 42nd {IEEE} Conf. Decision Control}, Maui, HI, Dec. 2003, pp. 4038--4053.

\bibitem{bib6}
G.~Martelli, ``A new optimum tuning method of PI controllers in first-order time-delay systems'' {\tt arXiv:math.OC/0703494v2 16 May 2007}.

\bibitem{bib7}
G.~Martelli, ``Beyond the PI controllers in first-order time-delay systems'' {\tt arXiv:math.OC/0705.3397v1 23 May 2007}.

\bibitem{bib8}
L.~S. Pontryagin, ``On the Zeros of Some Elementary Trascendental Functions,'' in \emph{Amer. Math. Soc. Transl.}, Ser.~2, Vol.~1, pp. 95--110, 1955.

\bibitem{bib9}
R.~Bellman, and K.~L. Cook, \emph{Differential-Difference Equations}, Academic Press Inc., London, 1963.

\end{thebibliography}
\end{document}